\documentclass[10pt,twoside,en,bysec]{matstud}
\usepackage{amssymb,amscd}
\newtheorem{theorem}{Theorem}[section]
\newtheorem{proposition}[theorem]{Proposition}
\newtheorem{lemma}[theorem]{Lemma}
\newtheorem{corollary}[theorem]{Corollary}
\theoremstyle{definition}
\newtheorem{remark}[theorem]{Remark}

\newtheorem{definition}[theorem]{Definition}

\newtheorem{problem}[theorem]{Problem}

\newcommand{\HP}{\hat P}
\newcommand{\Pt}{P_\tau}

\def\P{{\mathcal P}}
\newcommand{\K}{{\mathcal K}}
\newcommand{\I}{{\mathcal I}}
\newcommand{\V}{{\mathcal V}}
\newcommand{\W}{{\mathcal W}}

\newcommand{\IN}{{\Bbb N}}
\newcommand{\IR}{{\Bbb R}}

\newcommand{\Tych}{{\mathcal T}ych}
\newcommand{\Unif}{{\mathcal U}nif}
\newcommand{\supp}{\operatorname{supp}}
\newcommand{\cl}{\operatorname{cl}}
\newcommand{\conv}{\operatorname{conv}}
\newcommand{\id}{\operatorname{id}}
\newcommand{\pr}{\operatorname{pr}}
\newcommand{\diam}{\operatorname{diam}}
\newcommand{\e}{\varepsilon}
\newcommand{\bs}{\setminus}
\newcommand{\Metr}{{\mathcal M}etr}

\newcommand{\invlim}{\varprojlim}

\newcommand{\w}{\omega}

\begin{document}

\title{Topology of probability measure spaces, II:\break
barycenters of probability Radon measures and metrization of the functors $\P_\tau$ and $\HP$.
}
\shorttitle{Topology of probability measure spaces, II}
\author{T.O. Banakh}
\address{Ivan Franko National University of Lviv}
\subjclass{18B30, 28A33, 28C15, 54B30, 54H05}
\UDC{515.12}

\msyear{1995}
\msvolume{5}
\msnumber{1-2}
\mspages{88--106}
\received{01.10.1994}
\revised{07.06.2012}
\pageno{88}

\ruabstract{Т.Банах}{Топология пространств вероятностных мер, II:
барицентры вероятностных радоновских мер и метризация функторов $\P_\tau$ и $\HP$}%
{}

\enabstract{T.Banakh}{Topology of probability measure spaces, II:
barycenters of probability Radon measures and metrization of the
functors $\Pt$ and $\HP$}%
{This paper is a follow-up to the author's work [1] devoted to
investigation of the functors $\Pt$ and $\HP$ of spaces of probability
$\tau$-smooth and Radon measures. In this part, we study the barycenter map for spaces of Radon probability measures. The obtained results are applied to show that the functor $\HP$ is monadic in the category of metrizable spaces.
Also we show that the functors $\HP$ and $\Pt$ admit liftings to the
category $\mathcal{BM}etr$ of bounded metric spaces and also to the category
$Unif$ of uniform spaces, and investigate properties of those liftings.
}%

\maketitle

{\let\thefootnote\empty
\footnotetext{This paper was translated from Russian to English by Oles' Potyatynyk} }

This paper is a follow-up to the author's article [1], where the functors $\Pt$ and $\HP$ of spaces of probability
$\tau$-smooth and Radon measures were defined. We continue with the same numbering of chapters and propositions as in [1].

In this part, for spaces of Radon probability measures, the barycenter map is studied, this map will be used to show that the functor $\HP$ is monadic in the category of metrizable spaces.

Another question addressed in this article concerns the metrizability of the functors
$\Pt$ and $\HP$. In this direction, we shall show that the functors $\Pt$ and $\HP$ can be lifted to the category $\mathcal{BM}etr$ of bounded metric spaces and also to the category $\Unif$ of uniform spaces, and investigate the properties of those liftings. We should note that similar problems concerning the metrizability of the functor
$P_\beta$ were explored in [2]--[6].

\setcounter{section}{2}
\section{Barycenters and monadicity of the functor $\HP$}

Let $E$ be a locally convex vector space. By $E^{*}$ we shall denote the vector space of all continuous linear functionals on $E$, equipped with the topology of uniform convergence on weakly bounded subsets of $E$. (Let us recall that a subset $A\subset E$ is {\em weakly bounded} if each functional $f\in E^*$ is bounded on $A$).
 A subset $\K\subset E^*$ is be called {\em equicontinuous} if there exists a neighborhood $U\subset E$ of zero such that  $|f(x)|<1$ for any $x\in U$ and $f\in\K$. By $E^{**}$ we denote the space of continuous linear functionals on $E^*$ endowed with the natural topology, i.e. the topology of uniform convergence on equicontinuous subsets of $E^*$. It is well-known [7, c.182] that the canonical map $E\to E^{**}$ is an embedding with respect to the natural topology on  $E^{**}$, so from now on, we will treat the space $E$ as a subspace of $E^{**}$.

Let $X$ be a weakly bounded subset of a locally convex space $E$. Then for any measure $\mu\in\HP(X)$, the formula
$b_X(\mu)(f)=\mu(f|X)$, $f\in E^*$, defines a linear functional
$b_X(\mu)$ on $E^*$, which is called the {\em barycenter} of the measure $\mu$. The functional $b_X(\mu)$ is continuous, since the set $X$ is weakly bounded and $E^*$ carries the topology of uniform convergence on weakly bounded sets. Thus, we have defined the barycenter map $b_X:\HP(X)\to E^{**}$.

A map $f:C\to C'$ between convex sets is called {\em affine} if
$$f(t\,x+(1-t)y)=t\,f(x)+(1-t)f(y)\mbox{ \ for any $x,y\in C$ and $t\in [0,1]$.}$$
The following lemma is immediate:

\begin{lemma}\label{l3.1} The barycenter map $b_X:\HP(X)\to E^{**}$ is affine and 
$b_X(\delta_x)=x$ for all $x\in X$ (where $\delta_x$ denotes the Dirac measure, concentrated at the point $x$).
\end{lemma}

\begin{theorem}\label{t3.2} The map $b_X:\HP(X)\to E^{**}$ is continuous if and only if the set $X$ is bounded.
\end{theorem}

\begin{proof} Let us first assume that the map $b_X:\HP(X)\to
E^{**}$ is continuous. Fix $x\in X$. In order to prove that the set $X$ is bounded, it is sufficient to show that for any open set
$U\ni 0$ in $E^{**}$ there exists $n\in\IN$ such that $X\subset n\cdot U+x$.
Since the map
 $b_X:\HP(X)\to E^{**}$ is continuous, the set
$b^{-1}_X(U+x)$ is open in $\HP(X)$. Therefore, there exists  a $*$-weakly open set $W$ in the space $C_b^*(X)\supset \HP(X)$ such that $W\cap
\HP(X)=b_X^{-1}(U+x)$. Since the subset $\HP(X)\subset C_b^*(X)$ is bounded, there exists a number $n\in\IN$ such that $\frac1n
\HP(X)+\frac{n-1}n\delta_x\in W$. Since the set $\HP(X)$ is convex, $\frac1n \HP(X)+\frac{n-1}n\delta_x\in W\cap \HP(X)=b_X^{-1}(U+x)$.
Then for every $y\in X$, $\frac1n\delta_y+\frac{n-1}n\delta_x\in
W\cap\HP(X)=b_X^{-1}(U+x)$. By Lemma 3.1,
$b_X(\frac1n\delta_y+\frac{n-1}n\delta_x)=\frac1n y+\frac{n-1}nx\in U+x$.
This implies that $\frac1ny\in U+\frac1n x$ and $y\in n\, U+x$.
Thus, $X\subset n\,U+x$, i.e. the set $X$ is bounded.

Now we will prove that if $X$ is a bounded subset of $E$, then the map $b_X:\HP(X)\to E^{**}$ is continuous.

Fix a measure $\mu_0\in\HP(X)$ and a neighborhood 
$$U(b_X(\mu_0))=\{F\in E^{**}: |(F-b_X(\mu_0))(f)|<1\mbox{ for all $f$ from some equicontinuous set $\K\subset E^*$}\}\subset
E^{**}$$of its barycenter. Since the set $X$ is bounded and $\K$ is an equicontinuous family in $E^*$, there exists a constant $C\ge 1$ such that
$|f(x)|<C$ for all $x\in X$ и $f\in\K$.

Now let us recall that the measure $\mu_0$ on $X$ is Radon. Therefore, there exists a compact $K\subset X$ such that $\mu_0(K)>1-\frac1{12\,C}$.
Since the family $\K$ is equcontinuous, there exists a neighborhood of zero $U\subset E$ such that $|f(x)|<\frac1{24}$ for all $f\in\K$
and $x\in U$. Let $V=(U+K)\cap X$.

By the Ascoli theorem [8, 3.4.20], the family of maps $\{f|K\mid
f\in\K\}\subset C(K)$ is precompact in the topology of uniform convergence. Therefore, there exists a finite subset $\I\subset\K$ such that for any functional $f\in\K$ there exists a functional $g\in\I$ with
$|f(x)-g(x)|<\frac1{12}$ for any $x\in K$. Moreover,
$|f(x)-g(x)|<\frac16$ for all $x\in V$.

By Lemma [1, 1.19], the set $\{\mu\in\HP(X)\mid \mu(V)>1-\frac1{12\,C}\}$
is an open neighborhood of the measure $\mu_0$. Then the set $$W=\{\mu\in \HP(X)\mid
\mu(V)>1-\frac1{12\,C}, \; |\mu(g|X)-\mu_0(g|X)|<\frac1{12},\; g\in\I\}$$
is an open neighborhood of the measure $\mu_0$ в $\HP(X)$. We claim that $b_X(W)\subset U(b_X(\mu_0))=\{F\in E^{**}:
|(F-b_X(\mu_0))(f)|<1,\; f\in\K\}$. Indeed, let $\mu\in W$. Then for any
$f\in\K$ there exists a functional $g\in\I$ such that
$|f(x)-g(x)|<\frac1{12}$ for all $x\in K$. Then
$$
\begin{aligned}
|(b_X(\mu)-b_X(\mu_0))(f)|&=|\mu(f|X)-\mu_0(f|X)|\le\\
&\le |\mu(f|X)-\mu(g|X)|+|\mu(g|X)-\mu_0(g|X)|+|\mu_0(g|X)-\mu_0(f|X)|\le\\
&\le \int_V|f-g|d\mu+\int_{X\bs V}|f-g|d\mu+\frac1{12}+\int_K|f-g|d\mu_0+\int_{X\bs K}|f-g|d\mu_0\le\\
&\le \frac16+\frac{2C}{12\,C}+\frac1{12}+\frac1{12}+\frac{2C}{12\,C}<1.
\end{aligned}
$$
Thus, $b_X(\mu)\in U(b_X(\mu_0))$, which means that the map
$b_X:\HP(X)\to E^{**}$ is continuous. 
\end{proof}

\begin{remark}\label{r3.3} If $E$ is a closed subset of $E^{**}$, then $b_X(\HP(X))\subset E$ for any bounded set $X\subset E$. This implies from Lemma 3.1 and Theorem 3.2.
\end{remark}

\begin{definition} A convex subset $X$ of a locally convex space $E$ is called
\begin{itemize}
\itemsep=0pt
\parsep=0pt
\parskip=0pt
\item {\em barycentric} if $X$ is bounded in $E$ and $b_X(\HP(X))\subset X$;
\item {\em $\infty$-convex} if for any bounded sequence of points $(x_n)_{n\in\w}$ of $X$ and any sequence $(\lambda_n)_{n\in\w}$ of non-negative real numbers with $\sum_{n=0}^\infty \lambda_n=1$ the series $\sum_{n=0}^\infty\lambda_nx_n$ converges to some point $x\in X$;
\item {\em compactly convex} \ if for each compact subset $K\subset X$ its closed convex hull $\cl_X(\conv(K))\subset X$ is compact.
\end{itemize}
\end{definition}

\begin{proposition}\label{p3.5} Each compactly convex subset $X$ of a locally convex space is $\infty$-convex.
\end{proposition}

\begin{proof} To show that $X$ is $\infty$-convex, fix a bounded sequence $(x_n)_{n\in\w}$ in $X$ and a sequence $(\lambda_n)_{n\in\w}$ of non-negative real numbers with $\sum_{n\in\w}\lambda_n=1$. Choose any number $m\in\w$ with $\lambda_m>0$. Then $\sum_{n\ne m}\lambda_n=1-\lambda_m<1$ and we can choose a sequence of positive real numbers $\e_n<1$, $n\in\w$, such that $\lim_{n\to\infty}\e_n=0$ and $\sum_{n\ne m}\lambda_n/\e_n\le 1$. Let $\lambda_n'=\lambda_n/\e_n$ for $n\ne m$ and $\lambda_m'=1-\sum_{n\ne m}\lambda_n'$. 

For every $n\in\w$ consider the point $x_n'=(1-\e_n)x_m+\e_n x_n\in X$ and observe that the sequence $(x_n')_{n\in\w}$ tends to $x_m=x'_m$. Then the set $K=\{x_n'\}_{n\in\w}\subset X$ is compact and so is its closed convex hull $\cl_X(\conv(K))$ in $X$. Now observe that
$$
\begin{aligned}
\sum_{n\in\w}\lambda_nx_n&=\lambda_mx_m+\sum_{n\ne m}\lambda_n'\e_nx_n=\\
&=\lambda_mx_m-\sum_{n\ne m}\lambda_n'(1-\e_n')x_m+\sum_{n\ne m}\lambda_n'((1-\e_n)x_m+\e_nx_n)=\\
&=\big(\lambda_m-\sum_{n\ne m}\lambda_n'+\sum_{n\ne m}\lambda_n'\e_n'\big)x_m+\sum_{n\ne m}\lambda_n'x_n'=\\
&=\big(\lambda_m-\sum_{n\ne m}\lambda_n'+\sum_{n\ne m}\lambda_n\big)x_m+\sum_{n\ne m}\lambda_n'x_n'=\\
&=\big(\sum_{n\in\w}\lambda_n-\sum_{n\ne m}\lambda_n'\big)x_m+\sum_{n\ne m}\lambda_n'x_n'=\\
&=\big(1-(1-\lambda_m')\big)x_m+\sum_{n\ne m}\lambda_n'x_n'=\sum_{n\in\w}\lambda_n'x_n'\in\cl_X(\conv(K))\subset X,
\end{aligned}
$$
witnessing that the set $X$ is $\infty$-convex.
\end{proof}

\begin{theorem}\label{t3.6} A bounded convex subset $X$ of a locally convex space $E$ is barycentric if and only if it is compactly convex.
\end{theorem}

\begin{proof} To prove the ``only if'' part, assume that the set $X$ is barycentric. Then for each compact subset $K\subset X$ its probability measure space $P(K)\subset \HP(X)$ is compact [9,
VII.3.5]. Since the map $b_X:\HP(X)\to X$ is continuous, $b_X(P(K))$ is compact in $X$ and so is the closed convex hull $\cl_X(\conv(K))=b_X(P(K))$ of $K$ in $X$. This means that the set $X$ is compactly convex.

To prove the ``if part'', assume that the convex set $X\subset E$ is compactly convex. By Proposition~\ref{p3.5}, the set $X$ is $\infty$-convex. To prove that $X$ is barycentric, fix a Radon measure $\mu\in\HP(X)$ on $X$. If
the support $\supp(\mu)=\{x\in X\colon $ every neighborhood $U\ni x$ has a non-zero $\mu$-measure$\}$ of the measure $\mu$ is compact, then, by the compact convexity of $X$ its closed convex hull
$\cl_X(\conv(\supp(\mu)))$ is compact. In this case,
$b_X(\mu)\in\cl_X(\conv(\supp(\mu)))\subset X$. If the support of the measure $\mu$
is not compact, then there exists a sequence if compacta
$\emptyset=K_0\subset K_1\subset\dots\subset K_n\subset\dots\subset X$,
such that $0=\mu(K_0)<\mu(K_1)<\dots<\mu(K_n)<\dots <1$ and $\mu(K_n)>1-2^{-n}$,
$n\in\IN$. For every $n\in\IN$ let $\e_n=\mu(K_n\bs K_{n-1})$ and
$\mu_n\in P(K_n)\subset \HP(X)$ be the measure, defined by the formula
$\mu_n(A)=\mu((K_n\bs K_{n-1})\cap A)/\e_n$, where $A$ is a Borel subset of $X$. One can easily observe that $\sum_{n=1}^\infty \e_n=1$ and
$\mu=\sum_{n=1}^\infty \e_n\mu_n$. Since the map $b_X:\HP(X)\to
E^{**}$ is affine and continuous, $b_X(\sum_{n=1}^\infty
\e_n\mu_n)=\sum_{n=1}^\infty \e_nb_X(\mu_n)\in E^{**}$. By the compact convexity of $X$,
$b_X(\mu_n)\in X$, $n\in\IN$, and by the $\infty$-convexity of $X$,
$b_X(\mu)=\sum_{n=1}^\infty \e_nb_X(\mu_n)\in X$. So, the set $X$ is barycentric.
\end{proof}

Let us recall that a map $f:A\to B$ between topological spaces is called {\em quotient} if a set $U\subset B$ is open if and only if the set $f^{-1}(U)\subset A$ is open. 

\begin{proposition}\label{p3.7} For any barycentric set $X$ an a locally convex space, the berycenter map $b_X:\HP(X)\to X$ is surjective and quotient.
\end{proposition}

{\em The proof} follows the fact that for each point $x\in X$ the Dirac measure $\delta_x$ concentrated in $x$ has barycenter  $b_X(\delta_x)=x$.
\vskip5pt

Now let us prove several results on the preservation of compactly or $\infty$-convex sets by some operations. One can easily prove the following propositions.

\begin{proposition}\label{p3.8} If a convex subset $X$ of a locally convex space is compactly ($\infty$-) convex, then each closed convex subset of $X$ also has that property.
\end{proposition}

\begin{proposition}\label{p3.9} Let $X_i$, $i\in\I$, be convex subsets of locally convex spaces $E_i$, $i\in\I$, respectively. If all
$X_i$, $i\in\I$, are compactly ($\infty$-) convex, then their product
$\prod_{i\in\I}X_i\subset \prod_{i\in\I}E_i$ also has that property.
\end{proposition}

\begin{proposition}\label{p3.10} Let $X_i$, $i\in\I$, be convex subsets of a locally convex space $E$. If all $X_i$, $i\in\I$, are compactly ($\infty$-) convex, then their intersection $\bigcap_{i\in\I}X_i$ also has this property.
\end{proposition}

\begin{proposition}\label{p3.11} If a convex bounded subset $X$ of a locally convex space is $\infty$-convex and $X$ a countable union of compactly convex Borel subsets, then $X$ is compactly convex.
\end{proposition}

\begin{proof} By Theorem 3.6, it suffices to prove that $X$ is barycentric. So, let
$\mu\in\HP(X)$ be a Radon measure on $X=\bigcup_{n=1}^\infty X_n$, where
$X_n\subset X$ are compactly convex Borel subsets of $X$. Since the measure $\mu$ is Radon, for every $n,m\in\IN$ there exists a compact subset 
$K_n^m\subset X_n\bs (\bigcup_{i<n}X_i)$ such that
$\mu(K_n^m)\ge (1-2^{-m})\mu(X_n\bs (\bigcup_{i<n}X_i))$. For any
$n,m\in\IN$ let $\e_n^m=\mu(K^m_n\bs K_n^{m-1})$ (we assume that
$K_n^0=\emptyset$) and, if $\e_n^m>0$, then we can define a measure
$\mu_n^m\in\HP(X)$ by the following formula: $\mu_n^m(A)=\mu((K_n^m\bs
K_n^{m-1})\cap A)/\e_n^m$, where $A$ is a Borel subset of $X$. One can easily see that $\sum_{n,m=1}^\infty \e_n^m=1$ and $\sum_{n,m=1}^\infty
\e_n^m\mu_n^m=\mu$. Since the support of every measure $\mu_n^m$ is contained in the compact set $K_n^m\subset X_n$, $b_X(\mu_n^m)\in X_n$, $n,m\in\IN$
(let us recall that the sets $X_n$, $n\in\IN$, are compactly convex).
Let us show that $b_X(\mu)\in X$. Indeed, since the map
$b_X:\HP(X)\to E^{**}$ is affine and continuous,
$b_X(\mu)=b_X(\sum_{n,m=1}^\infty\e_n^m\mu_n^m)=\sum_{n,m=1}^\infty\e_n^m
b_x(\mu_n^m)\in X$ by the $\infty$-convexity of $X$. 
\end{proof}

\begin{proposition}\label{p3.12} For any $\infty$-convex subset $X$ of a locally convex space $E$ and any affine continuous map $T:E\to E'$ to a locally convex space $E'$, the image $T(X)$ is $\infty$-convex.
\end{proposition}

Krein's Theorem [7, IV.11.5] implies:

\begin{proposition}\label{p3.13} Every complete bounded convex subset of a locally convex set is barycentric.
\end{proposition}

\begin{corollary}\label{c3.14} Any convex compact subset of a locally convex space is barycentric.
\end{corollary}

Propositions 3.8 and 3.11 imply:

\begin{corollary}\label{c3.15} Any open bounded convex subset of a Banach space is barycentric.
\end{corollary}

Now we will consider some specific examples of barycentric sets. Let us recall that $Q=[-1,1]^\omega$ is a Hilbert cube, $s=(-1,1)^\omega$ is
its pseudo-interior, $\Sigma=\{(x_i)_{i=1}^\infty\in
Q:\sup_{i\in\IN}|x_n|<1\}$ is its radial-interior and
$\sigma=\{(x_i)_{i=1}^\infty\in s:x_i\ne0$ for a finite number of indices
$i\}$. All these spaces are assumed to be subsets of the locally convex space $\IR^\omega$. By Corollary 3.14, Hilbert cube $Q$ is barycentric, the pseudo-interior $s$ is barycentric by Proposition 3.10 and Corollary 3.15. Then, let us note that the radial-interior $\Sigma$ is an image of the unit open ball
$\{x\in l_\infty:\|x\|<1\}$ of the Banach space $l_\infty$ under the ``identity'' linear operator  $T:l_\infty\to \IR^\omega$.  Corollaries 3.12 and
 3.15 imply that the set $\Sigma$ is $\infty$-convex. Since $\Sigma$ is a union of convex compacta, Propositions 3.8 and Corollary 3.11 imply that the radial-interior $\Sigma$ is a convex barycentric set. Since $\Sigma^\omega=\bigcap_{n=1}^\infty (\Sigma^n\times
Q^{\omega-n})$, Proposition 3.9 implies that the convex set
$\Sigma^\omega$ is also barycentric. Further in the text we will see that there exist barycentric convex sets of any Borel complexity. On the other hand, the convex set $\sigma$ is not barycentric since it is not $\infty$-convex. For the same reason, the set
$\sigma\times Q$ is not barycentric neither. It is worth noting that the spaces
$\sigma\times Q$ and $\Sigma$ are homeomorphic.
\vskip5pt

In connection with barycentric sets the following problem appears naturally:

\begin{problem}\label{q3.13} Is the set $\HP(X)$ barycentric for every Tychonoff space $X$?
\end{problem}

We do not have the answer to this question. We will show, however, that the set
$\HP(X)$ is barycentric provided $\HP(X)=\Pt(X)$, or, if $X$ is a metrizable space. But let us first clarify a few points concerning the barycenter map
$b_{\HP(X)}$. Let $X$ be a Tychonoff space. In this case the barycenter map $b_{P(\beta X)}:P^2(\beta X)\to P(\beta X)$ is well-defined and coincides with the component $\psi_{\beta X}$ of the monad multiplication $\Bbb
P=(P,\delta ,\psi)$ (see [10]).
Since the space $\HP(X)$ can be naturally identified with a subspace of $P(\beta X)$, and $\HP^2(X)$  with a subspace of
$P^2(\beta X)$, by the definition we have that the barycenter map
$b_{\HP(X)}:\HP(\HP(X))\to (C_b^*(X))^{**}$ is a restriction of the map $b_{P(\beta X)}$. According to [1, 1.26], $b_{P(\beta
X)}(\Pt^2(X))\subset \Pt(X)$. If $\HP(X)=\Pt(X)$, then
$b_{\HP(X)}(\HP^2(X))=b_{P(\beta X)}(\HP^2(X))\subset b_{P(\beta
X)}(\Pt^2(X))\subset \Pt(X)=\HP(X)$, i.e. the set $\HP(X)$ is barycentric.

\begin{proposition}\label{p3.14} For every Tychonoff space $X$, the set $\HP(X)$ is $\infty$-convex.
\end{proposition}

\begin{proof} Given sequences $\{\mu_n\}_{n=1}^\infty\subset \HP(X)$ and
$\{t_n\}_{n=1}^\infty\subset [0,1]$, $\sum_{n=1}^\infty t_n=1$, we should prove that $\sum_{n=1}^\infty
t_n\mu_n\in\HP(X)$. Given any $\e>0$, choose $m\in\IN$ such that
$\sum_{n=1}^mt_n>1-\frac\e2$. For every $i\le m$ the Radon property of the measure $\mu_i$ yields a compact subset $K_i\subset X$ with $\mu_i(K_i)>1-\frac\e2$. Then $K=\bigcup_{n=1}^mK_n$ is a compact subset of measure $$\big(\sum_{n=1}^\infty t_n\mu_n\big)(K)=\sum_{n=1}^\infty
t_n\mu_n(K)\ge \sum_{n=1}^m
t_n\mu_n(K_n)>\big(1-\frac\e2\big)\sum_{n=1}^mt_n>\big(1-\frac\e2\big)^2>1-\e,$$ which means that
the measure $\sum_{n=1}^\infty t_n\mu_n$ is Radon. 
\end{proof}

\begin{theorem}\label{t3.15} For every metrizable space $X$, the set
$\HP(X)$ is barycentric.
\end{theorem}

\begin{proof} Let $X$ be a metrizable space and $c\,X$ be any compactification of $X$. By Theorem 3.6, it suffices to prove that $\HP(X)$ is compactly convex. For this purpose, take any compact subset  $\K\subset \HP(X)\subset P(c\, X)$. By Theorem 2 [11, c.96], for every $n\in\IN$ there exists a compact subset $K_n\subset X$ such that $\mu(K_n)>1-2^{-n}$ for any measure $\mu\in\K$.
By Lemma [1, 1.19], the set $\K_n=\{\mu\in P(c\, X)\colon \mu(K_n)\ge
1-2^{-n}\}$ is closed in $P(c\, X)$ for all $n\in\IN$. Then $\tilde
\K=\bigcap_{n=1}^\infty\K_n$ is a convex compact in $P(c\, X)$. Moreover, $\tilde \K\subset\HP(X)$. Obviously,
$\cl_{\HP(X)}(\conv(\K))=\cl_{\tilde \K}(\conv(\K))$ is a compact subset of $\HP(X)$.
\end{proof}

The question whether the set $\HP(X)$ is barycentric is closely connected with the question whether the functor $\HP$ is monadic (for the definition of the latter see [1, \S 1]).

\begin{theorem}\label{t3.16} The restriction $\HP:\Metr\to\Metr$ of the functor $\HP$ to the category $\Metr$ of metrizable spaces and their continuous maps is a monad.
\end{theorem}

\begin{proof} The functor $\HP$ can be included into the triple $\hat{\Bbb
P}=(\HP,\delta,\psi)$, where $\delta$ is the Dirac's transform, and the component
$\psi_X:\HP^2(X)\to \HP(X)$ of the multiplication $\psi$ coincides with the barycenter map $b_{\HP(X)}:\HP^2(X)\to \HP(X)$. Here we use Theorems 3.15 and
[1, 2.27], which says that for any metrizable space $X$ the set
$\HP(X)$ is both barycentric and metrizable. The fact that equalities $\psi\circ
T\delta=\psi\circ\delta T=\id_T$ and $\psi\circ \psi T=\psi\circ T\psi$ hold
follows from the corresponding equalities that hold for the monad ${\Bbb P}$.
\end{proof}

Closely connected with the notion of a monad is the notion of an algebra. Let us recall the definition. By a {\em $\mathbb T$-algebra} of a monad ${\Bbb T}=(T,\delta,\psi)$ on the category $\mathcal C$ we understand a pair $(X,\xi)$ consisting of an object $X$ of the category $\mathcal C$ and a morphism $\xi:T(X)\to X$ such that $\xi\circ\delta_X=\id_X$ and
$\xi\circ\psi_X=\xi\circ T(\xi)$. A {\em morphism} of $\Bbb T$-algebras $(X,\xi)$  and
$(X',\xi)$ is a morphism $f:X\to X'$ such that $\xi'\circ
T(f)=f\circ \xi$.

\begin{theorem}\label{t3.17} If $X$ is a convex bounded barycentric metrizable subset of a locally convex space, then the pair $(X,b_X)$
is an algebra of the monad $\HP:\Metr\to\Metr$.
In this case the barycentric map is a morphism of $\HP$-algebras.
\end{theorem}

\begin{proof} The equality $b_X\circ \delta_X=\id_X$ follows form Lemma
3.1. Let us show that $b_X\circ b_{\HP(X)}=b_X\circ\HP(b_X)$. Let us note that the maps $b_X\circ b_{\HP(X)}$ and $b_X\circ \HP(b_X)$ are affine and continuous. Also, if $\delta_\mu\in\HP^2(X)$ is a Dirac measure on
$\HP(X)$, then $b_X\circ b_{\HP(X)}(\delta_\mu)=b_X(\mu)$ and $b_X\circ
\HP(b_X)(\delta(\mu))=b_X(\delta(b_X(\mu))=b_X(\mu)$. As the convex hull of all Dirac measures is dense in $\HP^2(X)$, the sets $b_X\circ
b_{\HP(X)}$ and $b_X\circ \HP(b_X)$ coincide. Thus, $(X,b_X)$
is a $\HP$-algebra.

In order to prove the second statement of the theorem, observe that for any affine continuous map $f:X\to Y$ between barycentric subsets of locally convex spaces, the maps $b_Y\circ
\HP(f)$ and $f\circ b_X$ are continuous and affine. Taking into account that $b_Y\circ
\HP(f)(\delta _x)=b_Y(\delta _{f(x)})=f(x)=f\circ \delta _X(\delta _x)$
for any $x\in X$ and the convex hull of all Dirac measures is dense in
$\HP(X)$, we conclude that $b_Y\circ \HP(f)=f\circ b_X$, i.e. $f$ is a morphism of
$\HP$-algebras $(X,b_X)$ and $(Y,b_Y)$. The theorem is proved.
\end{proof}

It is known [10], [12] that the category of $\Bbb P$-algebras is isomorphic to the category ${\mathcal C}onv$ of convex compacta that lie in locally convex spaces, and their continuous affine maps.

\begin{problem}\label{q3.18} Is the category of algebras of the functor $\HP$ isomorphic to the category whose objects are convex bounded barycentric sets in locally convex spaces, and whose morphisms are their affine continuous maps.
\end{problem}

In light of the monadicity of the functor $\Pt:\Tych\to\Tych$, the following problem arises naturally:

\begin{problem}\label{q3.19} Describe the category of algebras of the monad
$\Pt:\Tych\to\Tych$.
\end{problem}

\begin{remark}\label{r3.20} Let us note that Theorem 3.20 implies that an algebra of the monad $\Pt$ is any pair $(X,b_X)$, where $X$ is a convex bounded barycentric subset of a locally convex space such that $\HP(X)=\Pt(X)$. In addition, affine continuous maps between such sets are morphisms of $\Pt$-algebras.
\end{remark}

\section{Lifting the functors $\Pt, \HP$ to categories of metric and uniform spaces}

In this section, for every bounded (pseudo)metric $d$ on a Tychonoff space $X$, we will construct a (pseudo)metric
$d_\tau$ on $\Pt(X)$ that extends the (pseudo)metric $d$. This will allow us to lift the functors $\Pt$ and $\HP$ to the category of bounded metric spaces, and also to the category of uniform spaces. All constructions will be made first for the functor $\HP$. Then, using the facts that the functor $\Pt$ preserves embeddings, and that $\Pt(X)=\HP(X)$ for any complete metric space, we will generalize the obtained results to the functor $\Pt$.

Let $d$ be a bounded pseudometric on a Tychonoff space $X$. For Radon measures $\mu, \eta \in \HP(X)$ let $$\hat
d(\mu,\eta)=\inf\{\lambda(d)\colon \lambda \in
\HP(X \times X), \;\HP(pr_1)(\lambda)=\mu, \;\HP(pr_2)(\lambda)=\eta\},$$ where
$pr_i:X\times X\to X$, $i=1,2,$ is the projection onto the $i$-th factor. Let us note that for any measures $\mu,\eta \in \HP(X)$ a measure $\lambda \in \HP(X\times X)$ with
$\HP(\pr_1)(\lambda)=\mu$ and $\HP(\pr_2)(\lambda)=\eta$ always exists. Indeed, for $\lambda$ we can  take the tensor product $\mu\otimes \eta$
of the measures $\mu$ and $\eta$ (see [1, 2.21 and  2.22]).

Formally, the metric $\hat d$ was introduced by L.V. Kantorovitch in [13] and for compact spaces was studied  by V.V. Fedorchuk [14]. Very recently, Yu.V. Sadovnichiy has shown [3] that for a metric space $(X,d)$, the metric $\hat d$ induces the subspace topology of the subspace $P_\beta(X)\subset \HP(X)$, consisting of measures with compact supports.

The following proposition is analogous to [14, \S4, Lemma 5].

\begin{lemma}\label{l4.1} For any Radon measures $\mu,\eta\in\HP(X)$, 
there exists a measure $\lambda\in\HP(X\times X)$ with
$\HP(pr_1)(\lambda)=\mu,\;\HP(pr_2)(\lambda)=\eta$ and $\hat
d(\mu,\eta)=\lambda(d).$
\end{lemma}

\begin{proof} Let $\{\lambda_n\}^\infty_{n=1}\subset\HP(X\times
X)$ be a sequence of Radon measures such that
$\HP(pr_1)(\lambda_n)=\mu,\;\HP(pr_2)(\lambda_n)=\eta,\;n\in\IN$, and
$\lim_{n\to\infty}\lambda_n(d)=\hat d(\mu,\eta).$ Let $cX$ be a compactification of the space $X$. Since the functor $\HP$ preserves embeddings, we can assume that
$\HP(X)\subset P(cX)$ and $\HP(X\times X)\subset P(cX\times cX)$. In addition,
$P(pr_1)(\lambda_n)=\mu$ and $P(pr_2)(\lambda_n)=\eta,\;n\in\IN$ (here
$pr_i:cX\times cX\to cX,\;i=1,2,$ is the projection onto the $i$-th factor).
As $P(cX\times cX)$ is a compact, there exists an accumulation point
$\lambda\in P(cX\times cX)$ of the sequence
$\{\lambda_n\}^\infty_{n=1}$.
Since the maps
$P(pr_i):P(cX\times cX)\to P(cX),\;i=1,2,$ are continuous,
$P(pr_1)(\lambda)=\mu$ and $P(pr_2)(\lambda)=\eta$. Let us show that
$\lambda\in\HP(X\times X)\subset P(cX\times cX)$. For this purpose, fix
$\varepsilon>0$  and find two compacts $K_1\subset X$ and $K_2\subset X$ such that
$\mu(K_1)>1-\frac\e2$ and $\eta(K_2)>1-\frac\e2$. Then $\lambda((cX\times
cX)\bs(K_1\times K_2))=\lambda\big(((cX\bs K_1)\times cX)\cap (cX\times(cX\bs
K_2))\big)\le\lambda((cX\bs K_1)\times cX)+\lambda(cX\times(cX\bs
K_2))=\mu(cX\bs K_1)+\eta(cX\bs K_2))<\frac\e2+\frac\e2=\e$. Thus, $\lambda(K_1\times K_2)>1-\e$, i.e. $\lambda\in\HP(X\times X)$.
By definition of the topology on $\HP(X\times X)$, the function $D:\HP(X\times X)\to \IR$, $D:\nu\mapsto\nu(d)$, is continuous. Since
$\lambda\in\HP(X\times X)$ is a limit point of the sequence
$\{\lambda_n\}^\infty_{n=1}\subset\HP(X\times X)$ and
$\lim_{n\to\infty}\lambda_n(d)=\lim_{n\to\infty}D(\lambda_n)=\hat
d(\mu,\eta)$,  $D(\lambda)=\lambda(d)=\hat d(\mu,\eta)$. 
\end{proof}

\begin{lemma}\label{l4.2} For any bounded continuous (pseudo)metric $d$
on a Tychonoff space $X$, the function $\hat d:\HP(X)\times\HP(X)\to\IR$
is a continuous (pseudo)metric on $\HP(X)$. Moreover,
$\diam(X,d)=\diam(\HP(X),\hat d)$.
\end{lemma}

\begin{proof}
Similarly to Lemma 6 from [14, \S 4], it can be proved that the function
$\hat d:\HP(X)\times \HP(X)\to\IR$ is a bounded (pseudo) metric
on $\HP(X)$ for any bounded continuous (pseudo) metric $d$ on $X$.

Let us check that the pseudometric $\hat d$ on $\HP(X)$ is continuous. For that purpose,
fix a measure $\mu_0\in\HP(X)$ and $\e>0$, and let
$D=\diam(X,d)+1=\sup\{d(x,x')\mid x,x'\in X\}+1$.
Let us fix a compact subset $K\subset X$ such that
$\mu_0(K)>1-\frac\e{8D}$. Let $\mathcal U=\{U_1,\dots,U_n\}$ be a finite cover of the compact
$K$ by open $\frac\e4$-balls
(with respect to the pseudometric $d$).
Since the measure $\mu_0$
is countable additive, there exists  an open set $V_1\subset \bar
V_1=\cl_X(V_1)\subset U_1$ such that $\mu_0(V_1)>\mu_0(U_1)-\frac\e{8Dn}$. By induction, construct open sets $V_k\subset X,\;1<k\le n$ such that
$V_k\subset\bar V_k\subset U_k\bs(\bigcup_{i<k}\bar V_i)$ and
$\mu(V_k)>\mu_0(U_k\bs\bigcup_{i<k}U_i)-\frac\e{8Dn}$. Let us note that the sets
$V_1,\dots,V_n$ are disjoint. Apart from that,
$\mu_0(V_1)+\mu_0(V_2)+\dots+\mu_0(V_n)>\mu_0(U_1)+\mu_0(U_2\bs
U_1)+\dots+\mu_0(U_n\bs\bigcup_{i<n}U_i)-\frac\e{8D}=\mu_0(\bigcup_{i\le
n}U_i)-\frac\e{8D}>1-\frac\e{8D}-\frac\e{8D}=1-\frac\e{4D}$. By [1, 1.19],
the set $U=\{\mu\in\HP(X)\mid\mu(V_i)>\mu_0(V_i)-\frac\e{4Dn},\;1\le
i\le n\}$ is an open neighborhood of the measure $\mu_0$.

We will show that $\hat d(\mu,\mu_0)<\e$ for any measure $\mu\in U$.
 Let $V_0=X\bs(\bigcup^n_{i=1}V_i)$. For
a Borel subset $B\subset X$ let $\mu|B$ be the measure on $X$ defined by 
$\mu|B(A)=\mu(B\cap A)$ for any Borel
set $A\subset X$. Let us consider the measure
$\lambda=\sum_{i,j=0}^n\alpha_{ij}\cdot(\mu_0|V_i)\times(\mu|V_j)$ on
$X\times X$, where coefficients $\alpha_{ij}\ge0$ are chosen to satisfy the following conditions for every $0\le i,j\le n$:
$$\sum_{j=0}^n\alpha_{ij}\mu(V_j)=1,\;\sum_{i=0}^n\alpha_{ij}\mu_0(V_i)=1\mbox{ \ and \ } \alpha_{ii}\mu_0(V_i)\mu(V_i)=\min\{\mu_0(V_i),\mu(V_i)\}.$$ One can easily check that $\lambda$ is a Radon measure on $X\times
X,\;\HP(pr_1)(\lambda)=\mu_0$ и $\HP(pr_2)(\lambda)=\mu$. Let us show that $\lambda(d)=\int_{X\times X\bs\bigcup_{i=1}^nV_i\times
V_i}d\;d\lambda+\int_{\bigcup_{i=1}^nV_i\times V_i}d\;d\lambda<\e$. Let us note that
$$\lambda(\bigcup_{i=1}^nV_i\times
V_i)=\sum_{i=1}^n\min\{\mu_0(V_i),\mu(V_i)\}>\sum_{i=1}^n(\mu_0(V_i)-
\frac\e{4Dn})>1-\frac\e{4D}-\frac\e{4D}=1-\frac\e{2D}.$$ Since
every set $V_i$ is contained in some $\frac\e4$-ball, for any $(x,y)\in\bigcup_{i=1}^nV_i\times V_i$ we get $d(x,y)<\frac\e2$. Then
$\lambda(d)<\frac\e{2D}D+(1-\frac\e{2D})\frac\e2<\e$. This implies that
$\hat d(\mu,\mu_0)<\e$. Thus, $\hat d$ is a continuous
pseudometric on $\HP(X)$.

The fact that $\diam(X,d)=\diam(\HP(X),\hat d)$ can be easily deduced from the definition of the pseudometric $\hat d$. 
\end{proof}

Let us recall that a metric on a topological space $X$ is called {\em compatible} if it generates  the initial topology on $X$.

\begin{lemma}\label{l4.3} If $d$ is a compatible bounded metric on $X$ then $\hat d$ is a compatible metric on $\HP(X)$.
\end{lemma}

\begin{proof}
Let $(X,d)$ be a bounded metric space. Lemma 4.2 implies that $\hat d$ is a continuous metric on $\HP(X)$. Let us show that the metric $\hat d$ induces the topology on $\HP(X)$. By 
Theorem 4 [11,II], the sets of the form $\{\mu\in\HP(X):|\mu(f)-\mu_0(f)|<1\}$, where $f$ runs over uniformly continuous bounded functions on $(X,d)$, constitute a subbase of open neighborhoods of the measure $\mu_0$. Let $f:(X,d)\to\IR$ be a uniformly continuous bounded function on $X$. Let us show that there exists an $\e>0$, such that for any measure $\mu\in\HP(X)$ if $\hat d(\mu,\mu_0)<\e$, then
$|\mu(f)-\mu_0(f)|<1$. Let $M=\sup\{f(x)\mid x\in X\}+1$. By the uniform continuity of $f$, there exists a
$0<\delta<\frac1{4M}$ such that $|f(x)-f(y)|<\frac12$ for any for any $x,y\in X$ with
$d(x,y)<\delta$. Let $\e=\delta^2$. We claim that  $|\mu(f)-\mu_0(f)|<1$ for any measure $\mu\in\HP(X)$ with $\hat d(\mu,\mu_0)<\e$. Lemma 4.1 implies the existence of a measure $\lambda\in\HP(X\times X)$ such that
$\HP(pr_1)(\lambda)=\mu$, $\HP(pr_2)(\lambda)=\mu_0$ and $\lambda(d)<\e$.
This implies that $\lambda(A)\le\delta$
for the set $A=\{(x,y)\in X\times X\mid
d(x,y)\ge\delta\}$. In this case
$|\mu(f)-\mu_0(f)|=|\int_{X\times X}f(x)d\lambda-\int_{X\times
X}f(y)d\lambda|\le\int_{X\times
X}|f(x)-f(y)|d\lambda\le\int_A|f(x)-f(y)|d\lambda+\int_{(X\times X)\bs
A}|f(x)-f(y)|d\lambda\le\frac12+2M\,\delta<1$. Thus, the topology on
$\HP(X)$, induced by the metric $\hat d$ coincides with the initial topology and the lemma is proved.
\end{proof}

Let $\mathcal C$ be a category and $U:\mathcal C\to\Tych$ be a "forgetful" functor. We say that a functor $F:\Tych\to\Tych$ can be lifted to the category $\mathcal C$ if there exists a functor $\tilde F:\mathcal C\to\mathcal
C$ such that $U\circ\tilde F=F\circ U$. Lemmas 4.2, 4.3 imply

\begin{theorem}\label{t4.4} The functor $\HP$ can be lifted to the category ${\mathcal {BM}}etr$
of bounded metric spaces and their continuous maps.
\end{theorem}

\begin{lemma}\label{l4.5} Let $K$ be a closed subset of a bounded metric space $(X,d)$, $\e>0$ and $\delta>0$. Then, for any measures $\mu,\eta\in\HP(X)$ with $\hat
d(\mu,\eta)\le\frac\e2\,\delta$ and $\mu(K)\ge1-\frac\e2$ we get
$\eta(O_\delta(K))\ge1-\e$, where $O_\delta(K)=\{x\in X\colon
d(x,K)\le\delta\}$ is a $\delta$-neighborhood of the set $K$.
\end{lemma}

\begin{proof} Let $\mu,\eta\in\HP(X)$ be measures such that $\hat
d(\mu,\eta)\le\frac\e2\,\delta$ и $\mu(K)\ge1-\frac\e2$. By Lemma 4.1,
there exists a measure $\lambda\in\HP(X\times X)$ such that
$\HP(pr_1)(\lambda)=\mu,\;\HP(pr_2)(\lambda)=\eta$ and $\lambda(d)=\hat
d(\mu,\eta)$. Let $A=\{(x,y)\in X\times X\colon d(x,y)\le\delta,\;x\in
K\}$. One can easily see that $\eta(O_\delta(K))\ge\lambda(A)$. Let us note that
$(X\times X)\bs A=\{(x,y)\in X\times X\mid d(x,y)>\delta\}\cup
pr_1^{-1}(X\bs K)$. As $\lambda(d)\le\frac\e2\,\delta$,
$\lambda(\{(x,y)\in X\times X\mid d(x,y)>\delta\})\le\frac\e2$. The fact that
$\mu(K)\ge1-\frac\e2$ и $\HP(pr_1)(\lambda)=\mu$ implies
$\lambda(pr_1^{-1}(X\bs K))\le\frac\e2$. In this case $\lambda((X\times X)\bs
A)\le\frac\e2+\frac\e2$. Therefore,
$\eta(O_\delta(K))\ge\lambda(A)\ge1-\e$. 
\end{proof}

\begin{theorem}\label{t4.6} If $d$ is a complete bounded metric on $X$, then
$\hat d$ is a complete bounded metric on $\HP(X)$.
\end{theorem}

\begin{proof} Let $cX$ be any compactification of $X$ and
$\{\mu_n\}_{n=1}^\infty\subset\HP(X)$ be a $\hat d$-Cauchy sequence.
Then there exists a measure $\mu\in P(cX)$ which is a limit point of the sequence $\{\mu_n\}_{n=1}^\infty\subset\HP(X)\subset P(cX)$.
Let us show that $\mu\in\HP(X)$. Fix $\e>0$. We will construct inductively a sequence $\{k(n)\}_{n=1}^\infty\subset\IN$ of numbers and an increasing sequence 
$(K_n)_{n\in\w}$ of compact subsets of $X$ such that for any $n\in\IN$ and any $k,m\ge k(n)$ we get $\hat
d(\mu_k,\mu_m)\le\frac\e2\,2^{-n}$ and $\mu_{k(n)}(K_n)\ge1-\frac\e2$.
Since the space $(X,d)$ is complete, the intersection $K=\bigcap_{n=1}^\infty
O_{2^{-n}}(K_n)\subset X$ is compact [8, 4.3.29]. For every $n\in\IN$ let $C_n=\cl_{cX}(O_{2^{-n}}(K_n))$. One can easily check that
 $K=\bigcap_{n=1}^\infty C_n$. By [1, 1.19], the set $\mathcal
K_n=\{\mu\in P(cX)\colon\mu(C_n)\ge1-\e\}$ is closed in  $P(cX)$. Lemma
4.5 implies that for any $n\in\IN$ and $k\ge k(n)$ we get $\mu_k\in\mathcal K_n$.
Consequently, $\mu\in\bigcap_{n=1}^\infty\mathcal K_n$ and hence 
$\mu(C_n)\ge1-\e$ for every $n\in\IN$. Taking into account that $K=\bigcap_{n=1}^\infty
C_n$, we conclude that $\mu(K)\ge1-\e$ and hence $\mu\in\HP(X)$. Being Cauchy, the
the sequence $\{\lambda_n\}_{n=1}^\infty\subset\HP(X)$ converges to its accumulation point $\mu\in\HP(X)$.
\end{proof}

Now we will investigate the action of the functor $\HP$ on various classes of maps connected with the metric structure. 

\begin{proposition}\label{p4.7} The functor $\HP$ preserves the class of isometric embeddings.
\end{proposition}

\begin{proposition}\label{p4.8} The functor $\HP$ preserves non-expanding mappings.
\end{proposition}

\begin{proof} Let $f:(X,d)\to(Y,\varrho)$ be a non-expanding map and $\mu,\eta\in\HP(X)$. Let us show that
$\hat\varrho(\HP(f)(\mu),\HP(f)(\eta))\le\hat d(\mu,\eta)$. By Lemma 4.1,
there exists a measure $\lambda\in\HP(X\times X)$ such that
$\HP(pr_1)(\lambda)=\mu,\;\HP(pr_2)(\lambda)=\eta$ and $\lambda(d)=\hat
d(\mu,\eta)$. Then the measure $\HP(f\times f)(\lambda)\in\HP(Y\times Y)$
satusfies the following conditions: $\HP(pr_1)(\HP(f\times
f)(\lambda))=\HP(f)(\mu)$ and $\HP(pr_2)(\HP(f\times
f)(\lambda))=\HP(f)(\eta)$. Therefore,
$\hat\varrho(\HP(f)(\mu),\HP(f)(\eta))\le\HP(f\times
f)(\lambda)(\varrho)=\lambda(\varrho\circ(f\times f))$. Since the map
$f$ is non-expanding, $\varrho(f(x),f(y))\le d(x,y)$ for every $(x,y)\in X\times
X$. Thus,
$\hat\varrho(\HP(f)(\mu),\HP(f)(\eta))\le\lambda(\varrho\circ(f\times f))\le
\lambda(d)=\hat d(\mu,\eta)$. 
\end{proof}

Now we will deal with the question of extension of pseudometrics from a Tychonoff space $X$ to the space $\Pt(X)$ of $\tau$-smooth probability measures on $X$.

Let $p$ be a continuous pseudometric on a Tychonoff space $X$. By $(X_p,d_p)$ we denote the metric space induced by the pseudometric $p$ and by $\pi:X\to X_p$ the corresponding projection. Let $(X'_p, d'_p)$ be the completion of the metric space $(X_p,d_p)$.
As $X'_p$ is a complete metric space,
$\Pt(X_p')=\HP(X_p')$. Let us consider a map $\Pt(\pi):\Pt(X)\to
\Pt(X_p)\subset \Pt(X_p')=\HP(X_p')$ and define a pseudometric $p_\tau$ on
$\Pt(X)$ by the formula $p_\tau(\mu,\eta)=\hat
d'_p(\Pt(\pi)(\mu),\Pt(\pi)(\eta))$, $\mu,\eta\in \Pt(X)$, where $\hat d'_p$
is a metric on the space $\HP(X'_p)=\Pt(X'_p)$ induced by the metric
$d'_p$ on $X'_p$.

Since the functors $\Pt$ and $\HP$ preserve the class of embeddings, Lemmas 4.2,
4.3, and the definition of the pseudometric $p_\tau$ imply

\begin{proposition}\label{p4.9} If $p$ is a continuous bounded pseudometric on a Tychonoff space $X$, then $p_\tau$ is a continuous pseudometric on $\Pt(X)$. Moreover, $\diam(X,p)=\diam(\Pt(X),p_\tau)$.
\end{proposition}

\begin{proposition}\label{p4.10} If $d$ is a compatible bounded metric on the space $X$, then $d_\tau$ is a compatible bounded metric on $\Pt(X)$.
\end{proposition}

\begin{theorem}\label{t4.11} The functor $\Pt$ can be lifted to the category $\mathcal{BM}etr$
of bounded metric spaces and their continuous maps.
\end{theorem}

\begin{remark}\label{r4.12} For any bounded pseudometric $d$ on $X$
and any Radon measures
$\mu,\eta\in\HP(X)\subset\Pt(X)$ the distance $d_\tau(\mu,\eta)$ coincides with the distance $\hat d(\mu,\eta)$, defined at the beginning of this section. This follows from the fact that the pseudometrics $d_\tau$ and $\hat d$ are continuous on $\HP(X)$ and the equality $\hat d(\mu,\eta)=d_\tau(\mu,\eta)$ holds for measures $\mu,\eta\in\HP(X)$ with finite supports.
\end{remark}

Since for complete metric spaces the spaces of Radon and $\tau$-smoothe measures coincide, Theorem
4.6 implies:

\begin{theorem}\label{t4.13} If $d$ is a complete bounded metric on $X$, then
$d_\tau$ is a complete bounded metric on $\Pt(X)$.
\end{theorem}

Using Propositions 4.7, 4.8, we can prove:

\begin{proposition}\label{p4.14} The functor $\Pt$ preserves the class of isometric embeddings.
\end{proposition}

\begin{proposition}\label{p4.15} The functor $\Pt$ preserves non-expanding mappings.
\end{proposition}

For maps $f,g:Y\to X$ and a bounded pseudometric $d$ on $X$
let $d(f,g)=\sup\{ d(f(y),g(y)): y\in Y\}$.

\begin{proposition}\label{p4.16} For any bounded pseudometric
$d$ on a Tychonoff space $X$ and any maps $f,g:Y\to X$ we have that
$d(f,g)=d_\tau(\Pt(f),\Pt(g))$.
\end{proposition}

\begin{proof} Without loss of generality, $d$ is a compatible metric on $X$. Let $(X',d')$ be a completion of the space $(X,d)$.
Let us note that $$
\begin{aligned}
d_\tau(\Pt(f),\Pt(g))&=\sup\{d_\tau(\Pt(f)(\mu),\Pt(g)(\mu))
\colon\mu\in\Pt(Y)\}\\
&\ge\sup\{d_\tau(\Pt(f)(\delta_y),\Pt(g)(\delta_y))
\colon y\in Y\}=\sup\{ d(f(y),g(y))\mid y\in Y\}=d(f,g).
\end{aligned}
$$
We will show that $d_\tau(\Pt(f),\Pt(g))\le d(f,g)$.
Fix a measure $\mu\in\Pt(Y)$. Consider a map $(f,g):Y\to
X\times X$ and $\Pt(f,g):\Pt(Y)\to\Pt(X\times X)$ and let
$\lambda=\Pt(f,g)(\mu)\in \Pt(X\times X)\subset \HP(X'\times X')$.
One can easily see that
$\Pt(pr_1)(\lambda)=\Pt(f)(\mu)$ and $\Pt(pr_2)(\lambda)=\Pt(g)(\mu)$.
Apart from that, $\supp(\lambda)\subset\{(x,x')\in X'\times X': d'(x,x')\le
d(f,g)\}$. Then $d_\tau(\Pt(f)(\mu),\Pt(g)(\mu))=
\hat d'(\Pt(f)(\mu),\Pt(g)(\eta))\le\int d'\;d\lambda\le
d(f,g)$. The proposition is proved.
\end{proof}

Now let us look into the connection between the metric and convex structures on $\Pt(X)$.

A pseudometric $\varrho$ on a convex subset $Y$ of a vector space will be called {\em convex} if for any $x,x',y,y'\in Y$ and
$t\in[0,1]$ we get
$$\varrho(tx+(1-t)y,tx'+(1-t)y')\le
t\varrho(x,x')+(1-t)\varrho(y,y').
$$ One can easily observe that balls are convex with respect to a convex pseudometric. Moreover, if $d$ is a convex pseudometric on $Y$, then for any $x,y\in Y$ and $t\in[0,1]$ we get $d(tx+(1-t)y,y)\le td(x,y)$.

One can easily prove the following

\begin{proposition}\label{p4.17} For any bounded pseudometric $d$
on a Tychonoff space $X$, the pseudometric $d_\tau$ on $\Pt(X)$ is convex.
\end{proposition}

\begin{proposition}\label{p4.18} Let $X$ be a convex barycentric subset of a locally convex space, and let $d$ be a convex bounded metric on $X$. Then, the barycenter map
$b_X:(\HP(X),\hat d)\to(X,d)$ is non-expanding.
\end{proposition}

\begin{proof} For Dirac measures $\delta_x,\delta_y\in\HP(X)$ we have that
$d(b_X(\delta_y),b_X(\delta_x))=d(x,y)$. From the fact that $d$ is convex it follows that for any $n\in\IN$, any points $x_i,y_i\in X,\;1\le i\le n$, and numbers $t_i\in[0,1],\;1\le i\le n$, with $\sum_{i=1}^nt_i=1$ we get
$$d\big(\sum_{i=1}^nt_ix_i,\sum_{i=1}^nt_iy_i\big)\le\sum_{i=1}^nt_id(x_i,y_i).$$
Let us show that for any measures $\mu,\eta\in P_\omega(X)\subset\HP(X)$ with finite supports, $d(b_X(\mu),b_X(\eta))\le\hat d(\mu,\eta)$.
Indeed, let $\lambda\in\HP(X\times X)$ be  a measure such that
$\HP(pr_1)(\lambda)=\mu,\;\HP(pr_2)(\lambda)=\eta$ and $\lambda(d)=\hat
d(\mu,\eta)$. The measures $\mu$ and $\eta$ can be written as the convex combinations 
$\mu=\sum_{i=1}^n\alpha_i\delta_{x_i}$ and $\eta=\sum_{j=1}^m\beta_j
\delta_{y_j}$ of Dirac measures. In this case the measure $\lambda$ can be written as
$\lambda=\sum_{i=1,j=1}^{n\;\;\;m}\gamma_{ij}\delta_{(x_i,y_j)}$ where
$\sum_{j=1}^m\gamma_{ij}=\alpha_i$ and $\sum_{i=1}^n\gamma_{ij}=\beta_j$ for $1\le i\le n,\;1\le j\le m$. Then
\begin{multline*}
d\big(b_X\big(\sum_{i=1}^n\alpha_i\delta_{x_i}\big),b_X \big(\sum_{j=1}^m\beta_j
\delta_{y_j}\big)\big)=d\big(\sum_{i=1}^n\alpha_ix_i,\sum_{j=1}^m\beta_jy_j\big)=\\
d\big(\sum_{i=1,j=1}^{n\;\;\;m}\gamma_{ij}x_i,\sum_{i=1,j=1}^{n\;\;\;m}
\gamma_{ij}y_j\big)\le\sum_{i=1,j=1}^{n\;\;\;
m}\gamma_{ij}d(x_i,y_j)=\lambda(d)=\hat d(\mu,\eta).
\end{multline*}
Thus, the inequality $d(b_X(\mu),b_X(\eta))\le\hat d(\mu,\eta)$ holds for all measures
$\mu,\eta$ from the dense subset
$P_\omega(X)\subset\HP(X)$. By the continuity of the map $b_X:\HP(X)\to X$, the inequality
 $d(b_X(\mu),b_X(\eta))\le\hat d(\mu,\eta)$
holds for all measures $\mu,\eta\in\HP(X)$. 
\end{proof}

A connection between the metric and monadic structures of the function $\Pt$ is described in the following:

\begin{theorem}\label{t4.19} For any bounded metric space
$(X,d)$ the identity $\delta_X:X\to\Pt(X)$ of the monad $\Pt$ is a closed isometric embedding, and the multiplication
$\psi_X=b_{\Pt(X)}:\Pt^2(X)\to\Pt(X)$ is a non-expanding map.
Moreover, for every point $x\in X$ and a measure $\bold\mu\in\Pt^2(X)$, we get 
$d_{\Pt^2(X)}(\delta_{\delta_x},\bold\mu)=d_{\Pt(X)}(\delta_x,\psi_X(\bold\mu))$.
\end{theorem}

\begin{proof} The first statement of the theorem follows from Propositions 4.17,
4.18 and the fact that the set $\HP^2(X)$ is dense in $P_\tau^2(X)$. The last statement of the theorem will follow as soon as we prove the equality 
$d_{\Pt^2(X)}(\delta_{\delta_x},\mathbf\mu)=d_{\Pt(X)}(\delta_x,\psi_X(\mathbf\mu))$
for any point $x\in X$ and any measure $\mu$  from the dense subset
$P_\omega(P_\omega(X))\subset\Pt(\Pt(X))$. In this case the support $\supp(\bold\mu)\subset P_\omega(X)$ of the measure $\bold\mu$ is finite and so is the set
$K=\{x\}\cup\{\supp(\eta)\mid\eta\in\supp(\bold\mu)\}\subset X$.
Observe that $\mathbf \mu\in P(P(K))\subset\Pt(\Pt(X))$. For the compact metric space 
$(K,d|K)$ the equality
$d_{P^2(K)}(\delta_{\delta_x},\bold\mu)=d_{P(K)}(\delta_x,\psi_K(\bold\mu))$
was proved in [22, \S 4, лемма 11]. Since the functor $\Pt$ preserves isometric embeddings, the last equality implies that
$d_{\Pt^2(X)}(\delta_{\delta_x},\bold\mu)=d_{\Pt(X)}(\delta_x,
\psi_X(\bold\mu))$. The theorem is proved.
\end{proof}

Now we will deal with the problem of lifting the functors $\HP,\Pt:\Tych\to\Tych$ to the category $\Unif$ of uniform spaces and their continuous maрs (see [8, chapter 8] for the basics of the theory of uniform spaces).

Let us recall that a map $f:(X,{\mathcal U})\to (Y,\V)$ between two uniform spaces is a {\em uniform homeomorphism} if $f$ is bijective and the maps $f$ and $f^{-1}$ are uniformly continuous. A map
$f:(X,{\mathcal U})\to (Y,\V)$ is called a {\em uniform embedding} if $f$ is a uniform homeomorphism of the space $(X,{\mathcal U})$ onto its image $f(X)\subset
(Y,\V)$.

We shall say that ${\mathcal U}$ is a {\em uniformity on a topological space} $X$ it generates the  topology of $X$. A pseudometric $d:X\times X\to\IR$ is called {\em $\mathcal U$-uniform} if for any $\e>0$ there exists a neighborhood of the diagonal $U\in {\mathcal U}$ such that $d(U)\subset [0,\e)$.

For a uniform space $(X,{\mathcal U})$ a family $\mathcal P$ of all $\mathcal U$-uniform pseudometrics on $X$  has the following properties:
\begin{itemize}
\itemindent=12pt
\item[(UP1)] if $\rho_1,\rho_2\in \mathcal P$, then $\max\{\rho_1,\rho_2\}\in \mathcal P$;
\item[(UP2)] for every pair $x,y$ of distinct points from $X$ there exists a
pseudometric $\rho\in \mathcal P$ such that $\rho(x,y)>0$.
\end{itemize}

Conversely, for every family $\mathcal P$ of pseudometrics on the set $X$ that satisfies (UP1)--(UP2), the family $\mathcal B$ of all neighborhoods of the diagonal of the form $\{(x,y)\mid \rho(x,y)<2^{-n}\}$, where $\rho\in P$ and
$n\in\IN$, forms a base for some uniformity on $X$.

Now, we will show how to construct a uniformity on the space $\Pt(X)$ given a uniformity ${\mathcal U}$ on the space $X$.

For every pseudometric $p$ on $X$ by $(X_p,d_p)$ we will denote the metric space induced by $p$, and by
$\pi_p:X\to X_p$ the corresponding quotient map. Every uniformity ${\mathcal U}$ on the space $X$ induces a uniform embedding
$e_{\mathcal U}:(X,{\mathcal U})\to \prod_{p\in\P({\mathcal U})}(X_p,d_p)$, defined by the formula
$e_{\mathcal U}(x)=(\pi_p(x))_{p\in\P({\mathcal U})}$, where $\P({\mathcal U})$ is the family of all $\mathcal U$-uniform bounded pseudometrics on $X$ (see [6, 8.2.2]).
The set $\P({\mathcal U})$ is equipped with the natural partial order $\le$.
One can easily observe that any two pseudometrics $\rho,\rho'\in\P({\mathcal U})$,
$\rho\le\rho'$, induce a non-expanding map $\pi_\rho^{\rho'}:(X_{\rho'}, d_{\rho'})\to (X_\rho,d_\rho)$. Evidently, the product $\prod_{p\in\P({\mathcal U})}X_p$ is homeomorphic (uniformly homeomorphic even) to the limit $\invlim X_\rho$ of the inverse system 
$\{X_\rho,\pi_\rho^{\rho'},\P({\mathcal U})\}$. Since the set $\mathcal P({\mathcal U})$ is directed (i.e. for any $\rho_1,\rho_2\in\P({\mathcal U})$ there exists $\rho\in\mathcal P({\mathcal U})$ such that
$\rho\ge \rho_1$ and $\rho\ge \rho_2$),  Proposition [1, 1.11] guarantees that the map $R:\Pt(\prod_{p\in\P({\mathcal U})}X_p)=\Pt(\invlim X_p)\to \invlim
\Pt(X_p)$ is a topological embedding. Therefore, the map
$R\circ \Pt(e_{\mathcal U}):\Pt(X)\to \invlim \Pt(X_p)$ is also
a topological embedding. On every space $\Pt(X_p)$ we
consider the uniformity induced by the metric $(d_p)_\tau$, which, by Proposition 4.10, is compatible with the topology of $\Pt(X_p)$.
Finally, let us equip the space $\Pt(X)$ with the uniformity ${\mathcal U}_\tau$
of the subspace of the inverse limit $\invlim \Pt(X_p)$ of uniform spaces $\Pt(X_p)$. The space $\HP(X)$ will be equipped with the uniformity
$\hat {\mathcal U}$ of the subspace of the uniform space $(\Pt(X),{\mathcal U}_\tau)$.

Thus for every uniformity ${\mathcal U}$ on the space $X$ we hace defined a uniformity ${\mathcal U}_\tau$ on the space $\Pt(X)$. One can easily observe that the uniformity ${\mathcal U}_\tau$ can be defined in a direct fashion. In particular, the base for that uniformity consists of the entourages 
$U=\{(\mu,\eta)\in\Pt(X)\times\Pt(X)\mid p_\tau(\mu,\eta)<2^{-n}\}$, where
$p\in\P({\mathcal U})$ and $n\in\IN$. Having chosen the roundabout way (using inverse limits), we didn't need to check that the family $\mathcal B$, indeed, satisfies the axioms of a base for a uniformity. (see [8, p.624]).
Moreover, we have also proved that if the uniformity ${\mathcal U}$ induces the topology of $X$, the uniformity ${\mathcal U}_\tau$ induces the topology ofn $\Pt(X)$.

\begin{proposition}\label{p4.20} Let  a family $\mathcal P$ of bounded pseudometrics on a uniform space $(X,{\mathcal U})$ satisfy the conditions (UP1)--(UP2). If the family of entourages $\mathcal B=\big\{\{(x,y)\in X^2: \rho(x,y)<2^{-n}\}:\rho\in \mathcal P,\;n\in\IN\big\}$ is a base of the uniformity ${\mathcal U}$, then the family $\mathcal B_\tau=\big\{\{(\mu,\eta)\in\Pt(X)^2:
\rho_\tau(\mu,\eta)<2^{-n}\}:\rho\in \mathcal P,\;n\in\IN\big\},$ is a base
for the uniformity ${\mathcal U}_\tau$ on $\Pt(X)$.
\end{proposition}

\begin{proof} First, let us note that since very pseudometric $\rho\in P$ is $\mathcal U$-uniform, every entourage $V\in \mathcal B_\tau$ belongs to ${\mathcal U}_\tau$ according to the definition of the uniformity ${\mathcal U}_\tau$.

Now fix a neighborhood of the diagonal $V\in{\mathcal U}_\tau$. The definition of the uniformity ${\mathcal U}_\tau$ implies that there exists a bounded pseudometric
$\rho\in\P({\mathcal U})$ such that $\{(\mu,\eta)\in \Pt(X)^2:
\rho_\tau(\mu,\eta)<1\}\subset V$. Let $D=\sup\{\rho(x,y)\mid x,y\in
X\}+1$ and $U=\{(x,y)\mid \rho(x,y)<\frac 13\}$. As $\mathcal B$ is a base for the uniformity ${\mathcal U}$, there exist $p\in \mathcal P$ and $n\in\IN$ such that
$W=\{(x,y):X^2\mid p(x,y)<2^{-n}\}\subset U$.

We are going to show that $\W=\{(\mu,\eta)\in\Pt(X)^2: p_\tau(\mu,\eta)<2^{-(n+1)}/D\}\subset
{\mathcal A}=\{(\mu,\eta)\in\Pt(X)^2: \rho_\tau(\mu,\eta)\le \frac 56\}$. Since the set $\mathcal A$ is closed in $\Pt(X)\times \Pt(X)$, and the set $\W$ is open, the inclusion $\W\subset \mathcal A$ follows from the inclusion
$(\mu,\eta)\in \mathcal A$ for any pair $(\mu,\eta)\in\W$ of measures with finite supports. Fix a pair $(\mu,\eta)\in\W$ of measures with finite supports. By Lemma 4.1 and Remark 4.12, there exists a measure $\lambda \in
P_\omega(X\times X)$ such that $\Pt(\pr_1)(\lambda )=\mu$, $\Pt(\pr_2)(\lambda
)=\eta$ and $\lambda (p)=p_\tau(\mu,\eta)=\int_{X\times X}pd\lambda
<2^{-(n+1)}/D$. Taking into account that the value of the function $p$ on the set $(X\times
X)\bs W$ is not less than $2^{-n}$, we conclude that $2^{-(n+1)}/D>\int_{X\times X}p\,d\lambda
\ge \int_{(X\times X)\bs W}p\,d\lambda \ge 2^{-n}\lambda ((X\times X)\bs
W)$, which implies that $\lambda ((X\times X)\bs U)\le \lambda ((X\times
X)\bs W)<2^n\cdot 2^{-(n+1)}/D=\frac 1{2D}$. Then,
$\rho_\tau(\mu,\eta)\le\lambda (\rho)=\int_{X\times X}\rho\, d\lambda
=\int_{(X\times X)\bs U}\rho\,d\lambda +\int_U\rho\, d\lambda \le D\cdot
\lambda ((X\times X)\bs U)+\frac 13\lambda (U)\le \frac 12+\frac 13=\frac
56$. The proposition is proved.
\end{proof}

\begin{corollary}\label{c4.21} Let $(X,d)$ be a bounded metric space and ${\mathcal U}$ be a uniformity on $X$ induced by the metric $d$.
Then the uniformity on $\Pt(X)$ induced by the metric $d_\tau$ coincides with the uniformity ${\mathcal U}_\tau$.
\end{corollary}

The idea of the proof of the following proposition belongs to Yu. Sadovnichiy.

\begin{proposition}\label{p4.22} If $f:(X,{\mathcal U})\to (Y,\V)$ is a uniformly continuous map, then the map $\Pt(f):(\Pt(X),{\mathcal U}_\tau)\to
(\Pt(Y),\V_\tau)$ is also uniformly continuous.
\end{proposition}

\begin{proof} Let $f:(X,{\mathcal U})\to (Y,\V)$ be a uniformly continuous map.
As we have mentioned before, the basis for the uniformity $(\Pt(Y),\V_\tau)$ is the family of sets of the form $U^\e_p=\{(\mu,\eta)\in\Pt(Y)\times \Pt(Y)\mid p_\tau(\mu,\eta)<\e\}$,
where $\e>0$ and $p\in\P({\mathcal U})$. Fix an $\e>0$ and $p\in\P({\mathcal U})$.
Since the map $f:(X,{\mathcal U})\to (Y,\V)$ is uniformly continuous, the pseudometric
$\rho=p\circ (f\times f)$ is uniform with regard to ${\mathcal U}$. Let us note that the definition of the pseudometric $p_\tau$ implies that for any $\mu,\eta\in\Pt(X)$,
\ $\rho_\tau(\mu,\eta)=p_\tau(\Pt(f)(\mu),\Pt(f)(\eta))$. Therefore,
$(\Pt(f)\times \Pt(f))^{-1}(U^\e_p)=U^\e_\rho=\{(\mu,\eta)\in\Pt(X)\times\Pt(X)
\mid \rho_\tau(\mu,\eta)<\e\}\in {\mathcal U}_\tau$, i.e. the map
$\Pt(f):(\Pt(X),{\mathcal U}_\tau)\to (\Pt(Y),\V_\tau)$ is uniformly continuous. The Proposition is proved.
\end{proof}

Using Proposition 4.22 we can prove:

\begin{theorem}\label{t4.23} The functors $\Pt:\Tych\to\Tych$ and $\HP:\Tych\to
\Tych$ can be lifted to the category $\Unif$ of uniform spaces and their uniformly continuous maps.
\end{theorem}

One can immediately prove

\begin{proposition}\label{p4.24} For every uniform space
$(X,{\mathcal U})$ the map $\delta :(X,{\mathcal U})\to(\HP(X),\hat {\mathcal U})\subset
(\Pt(X),{\mathcal U}_\tau)$, $\delta :x\mapsto \delta _x$, $x\in X$, is a uniform embedding.
\end{proposition}

Since every bounded uniformly continuous pseudometric on a subspace $Y$ of a uniform space $(X,{\mathcal U})$ can be extended to a $\mathcal U$-uniform bounded pseudometric on $X$ (см. [15],
or [8, 8.5.6]), we have the following

\begin{proposition}\label{p4.25} For every uniform embedding
$f:(X,{\mathcal U})\to(Y,\V)$ the map $\Pt(f):(\Pt(X),{\mathcal U}_\tau)\to
(\Pt(Y),\V_\tau)$ is also a uniform embedding.
\end{proposition}

It is well-known  (see [8, гл.8]) that  a complete uniform space is compact if and only if it is totally bounded. 

\begin{proposition}\label{p4.26} A uniform space $(X,{\mathcal U})$ is totally bounded if and only if the space $(\Pt(X), {\mathcal U}_\tau)$ is totally bounded too.
\end{proposition}

\begin{proof} Since the space $(X,{\mathcal U})$ can be uniformly embedded in $(\Pt(X),{\mathcal U}_\tau)$, the fact that
$(\Pt(X),{\mathcal U}_\tau)$ is totally bounded implies that the space $(X,{\mathcal U})$ is also totally bounded.

Now, if $(X,{\mathcal U})$ is a totally bounded uniform space,
its completion $(\tilde X,\tilde{\mathcal U})$ is compact. By Proposition 4.26,
the space $(\Pt(X),{\mathcal U}_\tau)$ can be uniformly embedded in the compact space $\Pt(\tilde X)=P(\tilde X)$ and, therefore, is a totally bounded uniform space. The Proposition is proved.
\end{proof}

\begin{remark}\label{r4.27} Despite the fact that the functor $\HP$ preserves complete metric spaces, it does not, generally speaking, preserve complete uniform spaces. This follows form the fact that for an uncountable set $A$ the uniform space $(\HP(\IR^A),\hat{\mathcal U})$ is not complete (here ${\mathcal U}$ is the product uniformity on $\IR^A$~). Indeed, let
$\mu\in\HP(\IR)$ be an arbitrary measure on $\IR$ with a non-compact support. For every finite $B\subset A$ let $\mu_B=\otimes_{\alpha \in
A}\mu_\alpha \in\HP(\IR^A)$, where
$$
\mu_\alpha =\begin{cases} \mu,&\mbox{ for } \alpha \in B;\\
\delta _0,&\mbox{ for } \alpha \notin B.\end{cases}
$$
It can be shown that $\{\mu_\alpha \}_{B\subset A}$ is a Cauchy net in
$\HP(\IR^A)$, which does not have an accumulation point.

The author is unaware of whether the functor $\Pt$ preserves complete uniform spaces${}^1$.
\end{remark}

Let us recall that a topological space $X$ is called ({\em Hewitt}) {\em Dieudonne complete} if
$X$ is homeomorphic to a closed subset of the Tychonov product of (separable) complete metric spaces.
 It is known [8] that every Lindelof space is Hewitt complete, and every metrizable space is Diudonne complete. 

\begin{problem}\label{q4.28}Do the functors $\Pt$ and $\HP$ preserve Hewitt or Diudonne complete spaces?\footnote{The answer to this question is given in the following articles:\newline
$\bullet$ V.V.Fedorchuk {\em On completeness-type properties of spaces of $\tau$-additive probability measures} (Russian) Vestnik Moskov. Univ. Ser. I Mat. Mekh. 1998, no. 5, 19--22, 71; transl. in Moscow Univ. Math. Bull. {\bf 53} (1998), no. 5, 20--23 (1999).\newline
$\bullet$ V.V. Fedorchuk V.V., {\em On a preservation of completeness of uniform spaces by the functor $P_\tau$}, Topology Appl. {\bf 91}:1 (1999) 25-45.\newline 
$\bullet$ V.V.Fedorchuk, {\em On the topological completeness of spaces of measures}. (Russian) Izv. Ross. Akad. Nauk Ser. Mat. 63 (1999), no. 4, 207--223; translation in Izv. Math. 63 (1999), no. 4, 827--843}  
\end{problem}

We will finish this section with the following simple statement, which follows from Proposition 4.17.

\begin{proposition}\label{p4.29} For every uniform space
$(X,{\mathcal U})$ the map $$\alpha :(\Pt(X),{\mathcal U}_\tau)\times
(\Pt(X),{\mathcal U}_\tau)\times [0,1]\to (\Pt(X),{\mathcal U}_\tau),\;\;\alpha (\mu,\eta,t)=t\mu+(1-t)\eta,$$ is uniformly continuous.
\end{proposition}


\end{document}